\documentclass[11pt,english,oneside]{amsart}
\usepackage[T1]{fontenc}
\usepackage[latin9]{inputenc}
\usepackage{geometry}
\usepackage{amssymb}
\usepackage{color}

\usepackage{graphicx,color}
\usepackage{amsmath, amssymb, graphics}

\usepackage{graphicx}

\makeatletter
\numberwithin{equation}{section} 
\numberwithin{figure}{section} 
  \@ifundefined{theoremstyle}{\usepackage{amsthm}}{}
  \theoremstyle{plain}
  \newtheorem{thm}{Theorem}[section]
  \theoremstyle{plain}
  
  \theoremstyle{plain}
  
  \theoremstyle{Remark}
  
  \theoremstyle{remark}
  
  \theoremstyle{plain}
  \newtheorem{lem}[thm]{Lemma}

\usepackage{geometry}

\smallskip
\def\<{{\langle }}
\def\>{{\rangle }}

\makeatother

\usepackage{babel}

\smallskip
\def\<{{\langle }}
\def\>{{\rangle }}

\makeatother

\begin{document}

\title[index minimal hypersurfaces]{On the index of minimal hypersurfaces of spheres}

\author{ Oscar M. Perdomo }

\date{\today}




\maketitle

\begin{abstract}

Let $M\subset S^{n+1}\subset\mathbb{R}^{n+2}$ be a compact minimal hypersurface of the $n$-dimensional Euclidean unit sphere. Let us denote by $|A|^2$ the square of the norm of the second fundamental form and $J(f)=-\Delta f-nf-|A|^2f$ the stability operator. It is known that the index (the number of negative eigenvalues of $J$) is 1 when $M$ is a totally geodesic sphere, and it is $n+3$ when $M$ is a Clifford minimal hypersurface. It has been conjectured that for any other minimal hypersurface, the index must be greater than $n+3$. One partial result for this conjecture states that if the index is $n+3$ and $M$ is not Clifford, then $\int_M |A|^2<n|M|$ where $|M|$ is the $n$ dimensional volume of $M$. Somehow this partial result states that if the index of $M$ is $n+3$  then the average of the function $|A|^2$ needs to be small. In this note we prove that this average cannot be very small. We will show that for any pair of positive numbers $\delta_1$ and $\delta_2$ with $\delta_1+\delta_2=1$,  if $\int_M |A|^2\le \delta_2 n|M|$ and $|A|^2(x)\le2n\delta_1$ for all $x\in M$, then the index of $M$ is greater than $n+3$. 
\end{abstract}

\section{Introduction}

The stability index of minimal hypersurfaces of spheres plays an important role in the understanding of the whole theory of minimal hypersurfaces. An examaple that illustrates this fact is the proof of the Willmore conjecture by Fernando C. Marques and Andr\'e Neves \cite{MN}. They not only proved the Willmore conjecture but also showed that the area of minimal surfaces on $S^3$ jumps from $4\pi$, when $M$ is a minimal totally geodesic sphere, to $2\pi^2$ when $M$ is the Clifford torus. That is, they proved that the area of any minimal surface that is neither totally geodesic nor Clifford must be greater than $2\pi^2$. One of many tools needed to show their result is the Theorem by Urbano \cite{U} where he shows that  if the index of a minimal surface is $5$, then this surfaces must be the Clifford torus. If we use the same techniques in Urbano's paper we obtain that if the index of an $M\subset S^{n+1}$ is $n+3$, then $\int_M|A|^2\le n|M|$ with equality only if $M$ is a Clifford hypersurface, \cite{P}. In this note we will show that for any pair of positive numbers $\delta_1$ and $\delta_2$ with $\delta_1+\delta_2=1$,  if $\int_M |A|^2\le n \delta_2|M|$ and $|A|^2(x)\le2n\delta_1$ for all $x\in M$, then the index of $M$ is greater than $n+3$. An observation that is related with the main result shown here is the one that, taking the coordinates of the Gauss map as test functions, we can easy show that if  $M$ has antipodal symmetry and the first nonzero eigenvalue of the Laplacian of $M$ is $n$, then  $\int_M|A|^2\ge n|M|$ with equality if and only if $M$ is Clifford.


\section{Proof of the main theorem}

Let us assume that $M\subset S^{n+1}\subset\mathbb{R}^{n+2}$ is a compact minimal hypersurface of the $n$-dimensional Euclidean unit sphere and let us denote by $\nu:M\to S^{n+1}$ a Gauss map.  For any $v\in\mathbb{R}^{n+2}$ let us define by $l_v:M\longrightarrow \mathbb{R}$ and $f_v:M\longrightarrow \mathbb{R}$ the functions given by $  l_v(x)=\<x,v\>$ and $  f_v(x)=\<\nu(x),v\>$. A direct verification using the Codazzi equations gives us that

\begin{eqnarray}
\Delta l_v=-nl_v,\quad \Delta f_v=-|A|^2f_v, \quad
\end{eqnarray}
and
\begin{eqnarray}
\nabla l_v(x)=v^T=v-f_v(x) \nu(x)-l_v(x) x, \qquad \nabla f_v(x)=-A(v^T)
\end{eqnarray}

Where $A_x:T_xM\longrightarrow T_xM$, given by $A_x(w)=-D_w\nu$ denotes the shape operator.

We will be denoting by $\rho:M\longrightarrow \mathbb{R}$ the first eigenfunction of the operator $J$. We have that $J(\rho)=\lambda_1 \rho$. It is known that if $M$ is not a totally geodesic sphere, then $\lambda_1\le -2n$ with equality only if $M$ is Clifford. See \cite{S} and \cite{P1}.

\begin{lem}Let $M\subset S^{n+1}\subset\mathbb{R}^{n+2}$ be a compact minimal oriented hypersurface of the $n$-dimensional Euclidean unit sphere. If $M$ is neither totally umbilical nor Clifford, then the vector space 

$$\Gamma = \{ a \rho+f_w+l_v: a\in\mathbb{R}, \, w,v \in \mathbb{R}^{n+2}\}$$

has dimension $2 n+5$.

\end{lem}

\begin{proof}
Let us argue by contradiction. If the dimension of $\Gamma$ is smaller than $2n+5$ then we can find $a,w,v$ with at least one of them not zero such that  $ a \rho+f_w+l_v$  is the zero function. Let us prove first that $a$ has to be zero. If $a$ is not zero then, taking the operator $J$ we obtain that

$$a\lambda_1 \rho=J(a \rho)= J(-l_v-f_w)=n f_w+|A|^2 l_v=|A|^2 l_v +n(-a\rho-l_v)$$

From the equation above we obtain that 

$$ a(\lambda_1+n)\rho = (|A|^2-n) l_v $$

which is a contradiction because the function $\rho$ is never zero while the function $l_v$ must change sign. To prove the case when $a=0$ is similar. If $l_v+f_w$ is the zero function then taking the laplacian we obtain that $-nl_v=|A|^2f_w$, therefore $nf_w=-nl_v=|A|^2f_w$. Therefore $(n-|A|^2) f_v$ is zero. This is a contradiction because we are assuming that $M$ is not totally geodesic or Clifford.

\end{proof}

\begin{thm}
Let $M\subset S^{n+1}\subset\mathbb{R}^{n+2}$ be a compact minimal oriented hypersurface of the $n$-dimensional Euclidean unit sphere. Let $\delta_1$ and $\delta_2$ be any pair of positive numbers such that  $\delta_1+\delta_2=1$. If  $\int_M |A|^2\le \delta_2 n |M|$ and $|A|^2(x)\le2n\delta_1$ for all $x\in M$, then the index of $M$ is greater than $n+3$.
\end{thm}

\begin{proof}
For any orthonormal basis $v_1,\dots v_{n+2}$ of $\mathbb{R}^{n+2}$ we have that 

$$  \sum_{i=1}^{n+2}\int_M\left(|A|^2-\delta_2 n\right) \, f_{v_i}^2=\int_M\left(|A|^2-\delta_2 n\right)<0$$

Therefore, we can pick a vector $v_0$ such that 

$$\int_M(|A|^2-n\delta_2)f_{v_0}^2\le 0$$

Once we have chosen this vector $v_0$, we define the $n+4$ dimensional space 

$$\Gamma_0=\{ a \rho+l_w+b f_{v_0}: a,b\in\mathbb{R}, \, w \in \mathbb{R}^{n+2}\quad \}$$

We will prove the theorem by showing that for any $f\in \Gamma_0$ we have that $\int_MfJ(f)<0$. Let us consider $f=a \rho+l_w+bf_{v_0}\in \Gamma_0$. A direct computation shows that 

\begin{eqnarray*}
\int_MfJ(f) &=& a^2 \lambda_1 \int_M\rho^2-\int_M|A|^2l_w^2-nb^2\int_Mf_{v_0}^2-2 b \int_M |A|^2 l_wf_{v_0}-2a \int_M|A|^2 \rho l_w\\ 
&<& -2 a^2 n \int_M\rho^2-\int_M|A|^2l_w^2-nb^2\int_Mf_{v_0}^2-2 b \int_M |A|^2 l_wf_{v_0}-2a \int_M|A|^2 \rho l_w\\ 
&\le&a^2\int_M\rho^2(\frac{|A|^2}{\delta_1}-2 n) -\int_M|A|^2(a\frac{\rho}{\sqrt{\delta_1}} + \sqrt{\delta_1} l_w)^2-\int_M |A|^2(b\frac{f_{v_0}}{\sqrt{\delta_2}}+\sqrt{\delta_2}l_w)^2\\
& & + b^2\int_M \left(\frac{|A|^2}{\delta_2}-n\right)f_{v_0}^2\\
&<&0
\end{eqnarray*}
\end{proof}
Let us finish this note by stating the following conjecture. In the same way we showed that the space $\Gamma$ is has dimension $2 n+5$ we can show that the space 

$$\Lambda=\{ a +f_w+l_{v}: a\in\mathbb{R}, \, w,v \in \mathbb{R}^{n+2}\quad \}$$

has dimension $2n+5$ when $M$ is neither totally umbilical nor Clifford.  We conjecture that this space contains a $n+4$ dimensional subspace where the bilinear form $(f,g)\longrightarrow \int_MfJ(g)$ is negative definite.


\begin{thebibliography}{1}

\bibitem{MN}  Fernando C. Marques, Andr\'e Nevex.  \emph{Min-Max theory and the Willmore conjecture}, Annals of Mathematics {\bf 179} (2014), 683-782.


\bibitem{U} Francisco Urbano \emph{Minimal surfaces with low index in the three dimensional sphere}, Proc. Amer. Math. Soc.  {\bf 108} (1990), 898-992.





\bibitem{P} Oscar Perdomo.  \emph{On the average of the scalar curvature of minimal hypersurfaces of spheres with low stability index}, Illinois J. of Math. {\bf 48}, (1968),   68-105.


\bibitem{P1} Oscar Perdomo.  \emph{First stability eigenvalue characterization of Clifford hypersurfaces}, Proc. Amer. Math. Soc. {\bf 130}, (2002),   3379-3384.


\bibitem{S} James Simons.  \emph{Minimal varieties in Riemannian manifolds }, Annals  of Mathematics. (2) {\bf 88}, (2004),  No. 2,  559-565.



\end{thebibliography}
\end{document}